\documentclass{article}
\usepackage{geometry}
\usepackage[british]{babel}
\usepackage{ae}
\usepackage{latexsym}
\usepackage{amsfonts}
\usepackage{amssymb}
\usepackage{amsthm}
\usepackage{graphicx}
\usepackage{multirow}
\usepackage{color}
\usepackage{amsmath}
\usepackage{amscd}
\usepackage{float}
\usepackage{hyperref}
\usepackage{tikz}
\usepackage{tikz-cd}
\usepackage{stmaryrd}

\newcounter{itemcounter}
\numberwithin{itemcounter}{section}

\newtheorem{thm}[itemcounter]{Theorem}
\newtheorem{lem}[itemcounter]{Lemma}
\newtheorem{defi}[itemcounter]{Definition}
\newtheorem{prop}[itemcounter]{Proposition}
\newtheorem{cor}[itemcounter]{Corollary}

\newtheorem{rem}[itemcounter]{Remark}

\newtheorem*{thm*}{Theorem}
\newtheorem*{con*}{Conjecture}
\newtheorem*{cor*}{Corollary}
\newtheorem*{ack*}{Acknowledgements}

\newcommand{\Syl}{\mathop{\rm Syl}\nolimits}
\newcommand{\Irr}{\mathop{\rm Irr}\nolimits}
\newcommand{\IBr}{\mathop{\rm IBr}\nolimits}
\newcommand{\Tr}{\mathop{\rm Tr}\nolimits}
\newcommand{\Mat}{\mathop{\rm Mat}\nolimits}
\newcommand{\Hom}{\mathop{\rm Hom}\nolimits}
\newcommand{\End}{\mathop{\rm End}\nolimits}
\newcommand{\Aut}{\mathop{\rm Aut}\nolimits}
\newcommand{\Out}{\mathop{\rm Out}\nolimits}
\newcommand{\Inn}{\mathop{\rm Inn}\nolimits}

\newcommand{\Rad}{\mathop{\rm Rad}\nolimits}
\newcommand{\rk}{\mathop{\rm rk}\nolimits}
\newcommand{\SF}{\mathop{\rm sf}\nolimits}
\newcommand{\OF}{\mathop{\rm f_\mathcal{O}}\nolimits}
\newcommand{\SOF}{\mathop{\rm sf_\mathcal{O}}\nolimits}
\newcommand{\mf}{\mathop{\rm mf}\nolimits}
\newcommand{\op}{\mathop{\rm op}\nolimits}
\newcommand{\nth}{\mathop{\rm th}\nolimits}
\newcommand{\cO} {\mathcal{O}}

\pgfqkeys{/tikz/commutative diagrams}{
  maps from/.style={
    /tikz/commutative diagrams/leftarrow,
    /utils/exec=\pgfsetarrowsend{cm |}}
}

\title{Towards Donovan's conjecture for abelian defect groups \footnote{This research was supported by the EPSRC (grant no. EP/M015548/1).} }
\author{Charles W. Eaton\footnote{School of Mathematics, University of Manchester, Manchester, M13 9PL, United Kingdom. Email: charles.eaton@manchester.ac.uk} and Michael Livesey\footnote{School of Mathematics, University of Manchester, Manchester, M13 9PL, United Kingdom. Email: michael.livesey@manchester.ac.uk}}
\date{4th June 2018}

\begin{document}

\maketitle

\begin{abstract}
We define a new invariant for a $p$-block of a finite group, the strong Frobenius number, which we use to address the problem of reducing Donovan's conjecture to normal subgroups of index $p$. As an application we use the strong Frobenius number to complete the proof of Donovan's conjecture for $2$-blocks with abelian defect groups of rank at most $4$ and for $2$-blocks with abelian defect groups of order at most $64$.

\end{abstract}

\section{Introduction}

Let $p$ be a prime and $(K,\cO,k)$ a $p$-modular system with $k$ algebraically closed. Donovan's conjecture states that for $P$ a fixed $p$-group, there are only finitely many blocks of $kG$ for finite groups $G$ with defect group $D \cong P$. One approach to the conjecture is the reduction to quasisimple groups followed by the application of the classification of finite simple groups. For example, in~\cite{ekks14} it was proved that Donovan's conjecture holds for elementary abelian $2$-groups in this way. The reason that this result cannot at present be extended to arbitrary abelian $2$-groups is that it is not known how Morita equivalence classes of blocks relate to those for blocks of normal subgroups of index $p$ in general. The special case of a split extension is treated in~\cite{kk96}, and this is used in the proof of Donovan's conjecture for elementary abelian $2$-groups.

The idea here is to extend the idea of~\cite{ke04} where, for blocks with a fixed defect group, Donovan's conjecture is shown to be equivalent to two conjectures: that the Cartan invariants are uniformly bounded and that the \emph{Morita Frobenius number} is uniformly bounded. Roughly the first conjecture says that the dimensions of basic algebras are bounded and the second that the size of the fields of definition of the basic algebras are bounded. We define the \emph{strong Frobenius number} $\SF(B)$ of a block $B$ of $kG$, which plays a similar role but can be shown to respect normal subgroups of index $p$ as follows:

\begin{thm*}
Let $G$ be a finite group and $N$ a normal subgroup of $G$ of index $p$. Now let $B$ be a block of $k G$ with abelian defect group $D$ and $b$ a block of $k N$ covered by $B$. Then $\SF(B)\leq\SF(b)$.
\end{thm*}

This is part of Theorem \ref{normal_index_p_theorem}, which also contains a version for blocks of $\cO G$.

We demonstrate the power of this definition by using it to prove Donovan's conjecture for abelian $2$-groups of rank at most $4$, and further to prove the conjecture for all abelian $2$-groups of order up to $2^6$. Whilst the strong Frobenius number is well-suited to the consideration of blocks defined over $\cO$, in our applications here we still have to make use of K\"ulshammer's reductions for Donovan's conjecture in~\cite{ku95}, which are only known over $k$.

\begin{thm*}
Let $P$ a finite abelian $2$-group of rank at most $4$ or of order at most $64$. Then Donovan's conjecture holds for $P$, that is, there are only finitely many Morita equivalence classes of blocks $B$ of group algebras $kG$ with defect group $D \cong P$.
\end{thm*}

The structure of the paper is as follows. In Section 2 we set up some notation and briefly discuss Morita equivalences. In Section 3 we introduce the strong Frobenius and strong $\cO$-Frobenius numbers, show that they are preserved under Morita equivalences between blocks defined over $\cO$, and show that they behave well under passing to a normal subgroup of index $p$. In Section 4 we consider the situation of a normal defect group containing the defect groups of the block under consideration and prove Corollary \ref{powers_to_centre:cor} which is necessary when proving Donovan's conjecture for blocks with abelian defect groups of order $64$. In Section 5 we give the preliminary results needed for the reduction and application of the classification of finite simple groups. In Section 6 we prove that Donovan's conjecture holds for abelian $2$-groups of rank at most $4$ and for abelian $2$-groups of order up to $64$.


\section{Morita equivalences}

In this section we briefly discuss some facts about Morita equivalences but we first set up some notation to be used throughout this paper. Let $G$ be a finite group and $B$ a block of $\cO G$. We write $kB$ for the block of $kG$ corresponding to $B$ and $KB$ for the $K$-subspace of $KG$ generated by $B$. We denote by $e_B\in\cO G$ the block idempotent corresponding to $B$ and $e_{kB}\in kG$ the block idempotent corresponding to $kB$. We denote by $\Irr(G)$ (respectively $\IBr(G)$) the set of irreducible characters (respectively Brauer characters) of $G$ and $\Irr(B)$ (respectively $\IBr(B)$) the subset of irreducible characters $\chi$ (respectively Brauer characters $\phi$) such that $\chi(e_B)\neq0$ (respectively $\phi(e_B)\neq 0$). For each $\chi\in\Irr(G)$ we denote by $e_\chi\in\overline{K}G$ the character idempotent corresponding to $\chi$, where $\overline{K}$ denotes the algebraic closure of $K$. Note that $\overline{K}B=\bigoplus_{\chi\in\Irr(B)}\overline{K}Ge_\chi$. If $A$ and $B$ are finitely generated $R$-algebras for $R \in \{ K,\cO ,k,\overline{K}\}$, we write $\mod(A)$ for the category of finitely generated $A$-modules and $A\sim_{\operatorname{Mor}} B$ if $\mod(A)$ and $\mod(B)$ are (Morita) equivalent as $R$-linear categories. A  basic idempotent $e$ of a finite dimensional $k$-algebra $A$ is an idempotent such that $eAe$ is a basic algebra of $A$.

It is well known that a Morita equivalence gives rise to an isomorphism of centres. For convenience we give some details of this here:

\begin{prop}\label{isom_centres}
Let $R\in\{K,\cO,k\}$.

(i) Let $A$ and $B$ be finite dimensional $R$-algebras and $M$ an $A$-$B$-bimodule inducing a Morita equivalence between $A$ and $B$. Then there exists an isomorphism $\phi:Z(A)\to Z(B)$ uniquely defined by $z.m=m.\phi(z)$ for all $z\in Z(A)$ and $m\in M$.

(ii) Suppose in particular that $e$ is a basic idempotent of $A$, $B=eAe$ and $M=Ae$, then the isomorphism of centres $\phi:Z(A)\to Z(eAe)$ is given by $z\mapsto eze$.

(iii) If $R=\cO $ and $A$ and $B$ are in fact both blocks of finite groups, then the induced isomorphism of centres $\overline{K}\otimes_{\cO }\phi:Z(\overline{K}A)\to Z(\overline{K}B)$ is defined by $e_\chi\mapsto e_\psi$, where $\chi\in\Irr(A)$ and $\psi\in\Irr(B)$ correspond under the Morita equivalence.
\end{prop}

\begin{proof}
(i) This is well known.

(ii) Certainly $zae=ae(eze)$ for all $z\in Z(A)$ and $a\in A$ and the claim follows from part (i).

(iii) $\overline{K}M$ induces a Morita equivalence between $\overline{K}A$ and $\overline{K}B$. So
\begin{align*}
\overline{K}M\cong\bigoplus_{\chi\in\Irr(A)}V_\chi\otimes_{\overline{K}}\Hom_{\overline{K}}(V_{\chi_M},\overline{K}),
\end{align*}
where $\chi_M\in\Irr(B)$ corresponds to $\chi$ through the Morita equivalence and $V_\chi$ and $V_{\chi_M}$ are simple left $\overline{K}A$ and $\overline{K}B$-modules corresponding to $\chi$ and $\chi_M$ respectively. The claim now follows by part (i).
\end{proof}

The following is proved in~\cite[Lemma 3.2]{kl17} but with the assumptions that $R=k$ and that $A$ and $B$ are a priori isomorphic as rings.

\begin{prop}
\label{morita_implies_iso}
Let $G$ and $H$ be finite groups, $R\in\{\cO,k\}$ and $A$ and $B$ blocks of $RG$ and $RH$ respectively. Suppose there exists an $A$-$B$-bimodule $M$ inducing a Morita equivalence between $A$ and $B$, so $M$ is a progenerator for $B^{\op}$ and $\End_{B^{\op}}(M)\cong A$ via this bimodule structure. Suppose further that $\dim_k(M\otimes_BS)=\dim_k(S)$ for all simple $kB$-modules $S$. Then $A$ and $B$ are isomorphic as $R$-algebras.
\end{prop}

\begin{proof}
As a right $B$-module
\begin{align*}
M\cong\Hom_R(P_1,R)^{\oplus a_1}\oplus\dots\oplus\Hom_R(P_t,R)^{\oplus a_t},
\end{align*}
where $P_1,\dots,P_t$ is a complete set of isomorphism classes of indecomposable projective left $B$-modules and $a_i\in\mathbb{N}$. Let $S_i$ be the simple $kB$-module corresponding to $P_i$. Note that
\begin{align*}
\dim_k(\Hom_R(P_i,R)\otimes_BS_j)=\delta_{ij}.
\end{align*}
Therefore
\begin{align*}
a_j=\dim_k(M\otimes_BS_j)=\dim_k(S_j),
\end{align*}
for all $1\leq j\leq t$. So $M\cong B$ as a right $B$-module. Therefore we have the following isomorphism of $R$-algebras:
\begin{align*}
B\cong\End_{B^{\op}}(B)\cong\End_{B^{\op}}(M)\cong A.
\end{align*}
\end{proof}



\section{The strong Frobenius number of a block}

In this section we define the strong Frobenius number of a block and we will see that the definition is natural when considering blocks defined over $\cO$. The main point of this definition is that it allows us to apply results in~\cite{el17} to prove Theorem \ref{normal_index_p_theorem}, crucial in reducing Donovan's conjecture to blocks of quasisimple groups.

\begin{defi}
If $G$ is a finite group and $g\in G$, then we define $g_p,g_{p'}\in G$ by $g=g_pg_{p'}$ with $g_p,g_{p'}\in C_G(g)$ and $g_p$ of order a power of $p$ and $g_{p'}$ of $p'$-order. Similarly if $\omega\in\overline{K}$ is a root of unity then we define $\omega_p$ and $\omega_{p'}$ by considering $\omega$ as an element of the finite group $\langle\omega\rangle$.
\end{defi}

Following~\cite{bk07} we define some field automorphisms of $k$ and use them to define twists of algebras.

\begin{defi}
Let $q$ be a, possibly zero or negative, power of $p$. We denote by $-^{(q)}:k\to k$ the field automorphism given by $\lambda\mapsto\lambda^{\frac{1}{q}}$. Let $A$ be a $k$-algebra. We define $A^{(q)}$ to be the $k$-algebra with the same underlying ring structure as $A$ but with a new action of the scalars given by $\lambda.a=\lambda^{(q)}a$, for all $\lambda\in k$ and $a\in A$. For $a\in A$ we define $a^{(q)}$ to be the element of $A$ associated to $a$ through the ring isomorphism between $A$ and $A^{(q)}$. For $M$ an $A$-module we define $M^{(q)}$ to be the $A^{(q)}$-module associated to $M$ through the ring isomorphism between $A$ and $A^{(q)}$.

Note that for $G$ a finite group, we have $kG\cong kG^{(q)}$ as we can identify $-^{(q)}:kG\to kG$ with the ring isomorphism:
\begin{align*}
-^{(q)}:kG&\to kG\\
\sum_{g\in G}\alpha_gg&\mapsto\sum_{g\in G}(\alpha_g)^qg.
\end{align*}
If $B$ is a block of $kG$ then we can and do identify $B^{(q)}$ with the image of $B$ under the above isomorphism.

By an abuse of notation we also use $-^{(q)}$ to denote the field automorphism of the universal cyclotomic extension of $\mathbb{Q}$ defined by $\omega_p\omega_{p'}\mapsto\omega_p\omega_{p'}^{\frac{1}{q}}$, for all $p^{\nth}$-power roots of unity $\omega_p$ and $p'^{\nth}$ roots of unity $\omega_{p'}$. If $\chi\in\Irr(G)$, then we define $\chi^{(q)}\in\Irr(G)$ to be given by $\chi^{(q)}(g)=\chi(g)^{(q^{-1})}$ for all $g\in G$ and we define $\varphi^{(q)}$ for $\varphi\in\IBr(G)$ in an analogous way. Note that if $S$ is a simple $kG$ module with Brauer character $\varphi$, then $S^{(q)}$ is a simple $kG$ module with Brauer character $\varphi^{(q)}$. If $B$ is a block of $\cO G$ with $\chi\in\Irr(B)$, then we define $B^{(q)}$ to be the block of $\cO G$ with $\chi^{(q)}\in\Irr(B^{(q)})$. By considering a simple $B$-module and the decomposition matrix of $\cO G$ we see that $(kB)^{(q)}=k(B^{(q)})$, in particular $B^{(q)}$ is well-defined.
\end{defi}

We now define the Frobenius number and the Morita Frobenius number of a finite dimensional $k$-algebra as in~\cite{bk07}. We also define a new invariant called the $\cO$-Frobenius number of a block of $\cO G$, where $G$ is a finite group.

\begin{defi}
The \textbf{Frobenius number} of a finite dimensional $k$-algebra $A$ is the smallest integer $n$ such that $A\cong A^{(p^n)}$ as $k$-algebras and the \textbf{Morita Frobenius number} $\mf(A)$ of $A$ is the smallest integer $n$ such that $A\sim_{\operatorname{Mor}}A^{(p^n)}$ as $k$-algebras. The \textbf{$\cO$-Frobenius number} $\OF(B)$ of a block $B$ of $\cO G$, where $G$ is a finite group, is the smallest integer $n$ such that $A\cong A^{(p^n)}$ as $\cO$-algebras.
\end{defi}

We now work towards a more restrictive notion of isomorphism between a block $B$ and $B^{(p^n)}$.

\begin{defi}
Let $G$ be a finite group. We define the $\mathbb{Z}$-linear map
\begin{align*}
\zeta_G:\mathbb{Z}G&\to\mathbb{Z}G\\
g&\mapsto g_pg_{p'}^{\frac{1}{p}}.
\end{align*}
\end{defi}

\begin{prop}\label{strong_isomorphism}
Let $G$ be a finite group. Then $\zeta_G$ restricted to $Z(\mathbb{Z}G)$ induces an algebra automorphism of $Z(\mathbb{Z}G)$. Furthermore the algebra automorphism induced on $Z(\overline{K}G)\cong \overline{K}\otimes_{\mathbb{Z}}Z(\mathbb{Z}G)$ sends $e_\chi$ to $e_{\chi^{(p)}}$ for all $\chi\in\operatorname{Irr}(G)$.
\end{prop}

\begin{proof}
We first show that the image of $Z(\mathbb{Z}G)$ under $\zeta_G$ is contained in $Z(\mathbb{Z}G)$. One can quickly check that $\zeta_G(hgh^{-1})=h\zeta_G(g)h^{-1}$ for all $g,h\in G$ and so the sum of elements in the conjugacy class of $G$ containing $g$ is sent to the sum of elements in the conjugacy class of $G$ containing $\zeta_G(g)$. Therefore, $\zeta_G$ maps $Z(\mathbb{Z}G)$ to itself.
\newline
\newline
Next note that $\zeta_G^{\circ n}$ is the identity on all of $\mathbb{Z}G$, where $n$ is the multiplicative order of $p\mod|G|_{p'}$, so $\zeta_G|_{Z(\mathbb{Z}G)}:Z(\mathbb{Z}G)\to Z(\mathbb{Z}G)$ is an isomorphism of $\mathbb{Z}$-modules.
\newline
\newline
Now if $g\in G$ and $\chi \in\operatorname{Irr}(G)$ then by considering the eigenvalues of $\rho(g)$ for $\rho$ a representation of $G$ affording $\chi$, we see that $\chi^{(p)}(\zeta_G(g))=\chi(g)$. Therefore if we extend $\zeta_G$ to the $\overline{K}$-linear isomorphism $Z(\overline{K}G)\to Z(\overline{K}G)$, then for all $\chi, \psi \in \operatorname{Irr}(G)$
\begin{align*}
\psi^{(p)}(\zeta_G(e_\chi))=\psi(e_\chi)=\chi(1)\delta_{\chi,\psi}=\chi^{(p)}(1)\delta_{\chi^{(p)},\psi^{(p)}}=\psi^{(p)}(e_{\chi^{(p)}})
\end{align*}
and since the $e_\chi$'s form a basis for $Z(\overline{K}G)$ we must have $\zeta_G(e_\chi)=e_{\chi^{(p)}}$. This clearly defines an algebra automorphism on $Z(\overline{K}G)$ and so also on $Z(\mathbb{Z}G)$.
\end{proof}

By an abuse of notation we use $\zeta_G$ to simultaneously denote the algebra automorphisms of $Z(KG)$, $Z(\cO G)$, $Z(kG)$, $Z(\overline{K}G)$ and $Z(\mathbb{Z}G)$. The following corollary is a trivial consequence of the above proposition.

\begin{cor}
If $G$ is a finite group and $B$ a block of $kG$ or $\cO G$, then $\zeta_G$ induces an isomorphism from $Z(B)$ to $Z(B^{(p)})$.
\end{cor}

\begin{lem}\label{lem:transfer}
Let $G$ be a finite group, $H\leq G$ and $R\in\{K,\cO ,k,\mathbb{Z},\overline{K}\}$. Then
\begin{align*}
\zeta_G\circ\Tr_H^G=\Tr_H^G\circ \zeta_H:Z(RH)\to Z(RG),
\end{align*}
where $\Tr_H^G:Z(RH)\to Z(RG)$ denotes the transfer map.
\end{lem}

\begin{proof}
If $g\in G$ then we denote by $C_g^G$ the conjugacy class of $G$ containing $g$ with a similar notation for $h\in H$. Now if $h\in H$, then 
\begin{align*}
&\zeta_G\circ\Tr_H^G\left(\sum_{x\in C_h^H}x\right)=\zeta_G\left([C_G(h):C_H(h)]\sum_{x\in C_h^G}x\right)=[C_G(h):C_H(h)]\sum_{x\in C_{\zeta_G(h)}^G}x\\
=&[C_G(\zeta_H(h)):C_H(\zeta_H(h))]\sum_{x\in C_{\zeta_H(h)}^G}x=\Tr_H^G\left(\sum_{x\in C_{\zeta_H(h)}^H}x\right)=\Tr_H^G\circ \zeta_H\left(\sum_{x\in C_h^H}x\right),
\end{align*}
where the third equality follows from the fact that $\zeta_G(h)=\zeta_H(h)$ and that $C_G(g)=C_G(\zeta_G(g))$ for all $g\in G$.
\end{proof}

\begin{defi}\label{strong_frobenius_definition}
Let $G$ be a finite group, $B$ a block of $kG$ (respectively $\cO G$) and $n\in\mathbb{N}$. We define a \textbf{strong Frobenius isomorphism} (respectively \textbf{strong $\cO$-Frobenius isomorphism}) of degree $n$ to be an isomorphism from $\phi:B\to B^{(p^n)}$ with the following properties:
\begin{enumerate}
\item $\phi|_{Z(B)}$ coincides with $\zeta_G^{\circ n}|_{Z(B)}$,
\item for every simple $B$-module $S$, the simple $B^{(p^n)}$-module associated to $S$ through $\phi$ is isomorphic to $S^{(p^n)}$.
\end{enumerate}
We define the \textbf{strong Frobenius number} $\SF(B)$ (respectively \textbf{strong $\cO$-Frobenius number} $\SOF(B)$) of $B$ to be the smallest $n\in\mathbb{N}$ such that there exists a strong Frobenius isomorphism (respectively strong $\cO$-Frobenius isomorphism) of degree $n$ from $B$ to $B^{(p^n)}$.
\end{defi}

\begin{rem}\label{remark_strong_frobenius}
Note that when $B$ is a block of $\cO G$ (as opposed to $kG$), condition (2) is unnecessary. This is because by Proposition~\ref{strong_isomorphism} condition (1) is equivalent to $\overline{K}\otimes_{\cO}\phi(e_\chi)=e_\chi^{(p^n)}$ for all $\chi\in\Irr(B)$. Condition (2) is then automatic by the surjectivity of the decomposition matrix.

If $B$ is a block of $\cO G$ and $\phi:B\to B^{(p^n)}$ is a strong $\cO$-Frobenius isomorphism of degree $n$, then $k\otimes_{\cO}\phi:kB\to kB^{(p^n)}$ is a strong Frobenius isomorphism of degree $n$ and so $\SF(kB)\leq\SOF(B)$.

Note that the degree of a strong Frobenius isomorphism is important. Say $B=B^{(p^n)}$, which will always be the case for some $n\in\mathbb{N}$, then a $k$-algebra isomorphism $B\to B^{(p^n)}$, could potentially be a strong Frobenius isomorphism of degree $n$ or $2n$ or indeed any other positive multiple of $n$. Furthermore the condition that it is a strong Frobenius isomorphism depends on this degree. However, when the context is clear we will omit the degree of the strong Frobenius isomorphism.
\end{rem}

Before we prove various results about the strong Frobenius number we set up some notation.

\begin{defi}
Let $G$ be a finite group and $B$ a block of $kG$ or $\cO G$. If we have a $k$-linear (respectively $\cO $-linear) isomorphism $\phi:B\to B^{(p^n)}$, then we define $\phi^{(p^n)}:B^{(p^n)}\to B^{(p^{2n})}$ by the following commutative diagram:
\begin{align*}
\begin{CD}
B
@>\phi>>
B^{(p^n)}\\
@VV(p^n)V @VV(p^n)V\\
B^{(p^n)}
@>\phi^{(p^n)}>>
B^{(p^{2n})},
\end{CD}
\end{align*}
where the vertical arrows are given by the natural ring isomorphisms between $B$ and $B^{(p^n)}$ (respectively $B^{(p^n)}$ and $B^{(p^{2n})}$). Now for $t\in\mathbb{N}$ we define $\phi^{\circ t}:B\to B^{(p^{tn})}$ to be the composition of isomorphisms:
\begin{align*}
\begin{CD}
\phi^{\circ t}:B
@>\phi>>
B^{(p^n)}
@>\phi^{(p^n)}>>
\dots
@>\phi^{(p^{(t-2)n})}>>
B^{(p^{(t-1)n})}
@>\phi^{(p^{(t-1)n})}>>
B^{(p^{tn})}.
\end{CD}
\end{align*}
\end{defi}

The following can be seen as a strong Frobenius version of~\cite[Lemma 3.3]{kl17}.

\begin{prop}
\label{sf_bound_with_iso}
Let $G$ be finite group and $B$ a block of $\cO G$ with defect group $D$ of order $p^d$. If we have an isomorphism $\phi:B\to B^{(p^n)}$, then $\phi^{\circ t}:B\to B^{(p^{tn})}$ is a strong $\cO$-Frobenius isomorphism, where $t=p^{2d}!$. Hence it induces a strong Frobenius isomorphism $kB\to kB^{(p^{tn})}$. In other words, $\SF(B)\leq\SOF(B)\leq t.\OF(B)$. In particular, if $B$ is a principal block, then $\SF(B)\leq\SOF(B)\leq t$.
\end{prop}

\begin{proof}
By the Brauer-Feit theorem~\cite{bf59} $|\Irr(B)|\leq p^{2d}$. Therefore, the isomorphism $\overline{K}B\to\overline{K}B^{(p^{tn})}$ induced by $\phi^{\circ t}$ satisfies $\phi^{\circ t}(e_\chi)=e_{\chi^{(p^{tn})}}$ for all $\chi\in\Irr(B)$ (as permutation of the characters must have order dividing $t$.) Therefore by Proposition~\ref{strong_isomorphism} $\phi|_{Z(B)}=\zeta_G^{\circ(tn)}|_{Z(B)}$. Therefore, $\phi$ satisfies condition (1) in Definition \ref{strong_frobenius_definition} and condition (2) is automatic by Remark~\ref{remark_strong_frobenius} so the result is proved.
\end{proof}

The strong Frobenius and strong $\cO$-Frobenius numbers are compatible with Morita equivalences of blocks defined over $\cO$.

\begin{prop}
\label{morita_O_proposition}
Let $G$ and $H$ be finite groups and $B$ and $C$ blocks of $\cO G$ and $\cO H$. Suppose we have a Morita equivalence between $B$ and $C$, then $\SF(kB)=\SF(kC)$ and $\SOF(B)=\SOF(C)$.
\end{prop}

\begin{proof}
We only prove the statement for strong Frobenius numbers. The statement for strong $\cO$-Frobenius numbers is similar but more direct.
\newline
\newline
Let $M$ be a $B$-$C$-bimodule inducing a Morita equivalent between $B$ and $C$ and let $\phi:kB\to kB^{(p^n)}$ be a strong Frobenius isomorphism. Now consider the following commutative diagram of Morita equivalences
\begin{align*}
\begin{CD}
\mod(kB)
@<kM\otimes_{kC}\_\_<<
\mod(kC)\\
@VV \Phi V @VV F V\\
\mod(kB^{(p^n)})
@>\Hom_k(M^{(p^n)},k)\otimes_{B^{(p^n)}}\_\_>>
\mod(kC^{(p^n)}),
\end{CD}
\end{align*}
where $\Phi$ is the Morita equivalence induced by $\phi$ and $F$ is the composition of the other three Morita equivalences. Now let $S$ be a simple $kC$-module and write $T:=kM\otimes_{kC}S$, a simple $kB$-module. Then by construction $(kM^{(p^n)})^*\otimes_{kB^{(p^n)}}T^{(p^n)}\cong S^{(p^n)}$, so we have $F(S)\cong S^{(p^n)}$ for all simple $kC$-modules $S$.
\newline
\newline
Recall from Proposition~\ref{isom_centres}(i) that all Morita equivalences induce isomorphisms of centres. Consider the isomorphisms
\begin{align*}
\begin{CD}
Z(B)
@<M\otimes_C\_\_<<
Z(C)\\
@VV \zeta_G V @VV f V\\
Z(B^{(p^n)})
@>\Hom_{\cO }(M^{(p^n)},\cO )\otimes_{B^{(p^n)}}\_\_>>
Z(C^{(p^n)}),
\end{CD}
\end{align*}
where the horizontal arrows are induced by the corresponding Morita equivalences and $f$ is obtained by composition of the other three isomorphisms. Note that tensoring these maps with $k$ gives the isomorphisms of centres from the previous commutative diagram. Now tensor these maps with $\overline{K}$ and study the image of $e_\chi$ for each $\chi\in\Irr(C)$. Using Proposition~\ref{isom_centres}(iii) for the horizontal arrows and Proposition~\ref{strong_isomorphism} for the vertical arrows we get:
\begin{center}
\begin{tikzcd}
  e_\psi \rar[maps from] \dar[maps to] & e_\chi \dar[maps to]{f} \\
e_{\psi^{(p^n)}} \rar[maps to] &e_{\chi^{(p^n)}},
\end{tikzcd}
\end{center}
for all $\chi \in \Irr(C)$, where $e_\psi$ is the image of $e_\chi$ under the isomorphism given by $\overline{K}\otimes_{\cO}M$.
So the induced isomorphism $Z(\overline{K}C)\to Z(\overline{K}C^{(p^n)})$ is given by $\zeta_H$ restricted to $Z(\overline{K}C)$ and so the induced map $Z(kC)\to Z(kC^{(p^n)})$ is also given by $\zeta_H$ restricted to $Z(kC)$, as desired. Hence we have a Morita equivalence $kC\to kC^{(p^n)}$ that preserves the dimensions of simples and so by Proposition~\ref{morita_implies_iso} we have an isomorphism. Furthermore this isomorphism has the two properties of Definition~\ref{strong_frobenius_definition}. We have shown that $\SF(kB)\geq \SF(kC)$ and by an identical argument we also have $\SF(kB)\leq \SF(kC)$ and the proposition is proved.
\end{proof}

Since the strong Frobenius number of a block is at least the Morita-Frobenius number,~\cite{ke04} implies that in order to prove Donovan's conjecture it suffices to bound the Cartan invariants and the strong Frobenius number in terms of the defect group, and this will be our strategy in Section \ref{rank4}. We summarize this here.

\begin{prop}
Let $P$ be a finite $p$-group. If there are natural numbers $c(P)$ and $m(P)$ such that if $G$ is a finite group and $B$ is a block of $\cO G$ with defect groups isomorphic to $P$, then the Cartan invariants of $B$ are at most $c(P)$ and $\SF(kB) \leq m(P)$ (or $\SOF(B) \leq m(P)$), then Donovan's conjecture holds for $P$ with respect to $k$. 

If Donovan's conjecture holds for $P$ with respect to $\cO$, then there are natural numbers $c(P)$ and $m(P)$ such that if $G$ is a finite group and $B$ is a block of $\cO G$ with defect groups isomorphic to $P$, then the Cartan invariants of $B$ are at most $c(P)$ and $\SF(kB) \leq \SOF(B) \leq m(P)$.
\end{prop}

\begin{proof}
The first part is a consequence of~\cite{ke04}, noting that $\mf(kB) \leq \SF(kB) \leq \SOF(B)$, and the second follows from Proposition \ref{morita_O_proposition}.
\end{proof}

The following is well known but we include it for convenience since we use it frequently.

\begin{lem}
\label{index_p_background_lem}
Let $G$ be a finite group, $N$ a normal subgroup of $G$ of index a power of $p$ and $R\in\{\cO,k\}$. Let $B$ be a block of $RG$ with defect group $D$ covering a block $b$ of $RN$. Then $B$ is the unique block of $RG$ covering $b$, $D \cap N$ is a defect group for $b$ and the stabilizer of $b$ in $G$ is $DN$.
\end{lem}

\begin{proof}
This follows from~\cite[5.5.6, 5.5.16]{nt88}. 
\end{proof}

The following was proved over $k$ in~\cite[Theorem 2.1]{el17} but we extend it to $\cO$.

\begin{thm}\label{theorem_unit}
Let $G$ be a finite group and $B$ a block of $\cO G$ with abelian defect group $D$. Let $N\vartriangleleft G$ such that $G=ND$ and $[G:N]=p$. Let $b$ be a block of $\cO N$ covered by $B$. Then $e_B=e_b$ and there exists a $G/N$-graded unit $a\in Z(B)$ such that $B=\bigoplus_{j=0}^{p-1}a^j\cO Ne_b$.
\end{thm}

\begin{proof}
By~\cite[Theorem 2.1]{el17}, there exists a $G/N$-graded unit $\overline{a}\in Z(kB)$ such that $kB=\bigoplus_{j=0}^{p-1}\overline{a}^jkNe_b$. Now lift $\overline{a}$ to $a\in Z(B)$. As $Z(B)$ is $G/N$-graded, we may assume $a$ is $G/N$-graded. Finally as $Z(B)$ is a local ring and $\overline{a}\notin\Rad(kB)$, then $a$ is a unit and the result is proved.
\end{proof}

The next result is key, and emphasizes the advantage in using the strong Frobenius number.

\begin{thm}\label{normal_index_p_theorem}
Let $G$ be a finite group and $N$ a normal subgroup of $G$ of index $p$. Now let $B$ be a block of $\cO G$ with abelian defect group $D$ and $b$ a block of $\cO N$ covered by $B$. Then $\SF(kB)\leq\SF(kb)$ and $\SOF(B)\leq\SOF(b)$.
\end{thm}

\begin{proof}
As for Proposition~\ref{morita_O_proposition} we only prove the statement for strong Frobenius numbers. The statement for strong $\cO$-Frobenius numbers is similar but easier once we replace~\cite[Theorem 2.1]{el17} with Theorem~\ref{theorem_unit}.
\newline
\newline
If $b$ is not $G$-stable, then $b$ is Morita equivalent to $B$ by the Fong-Reynolds theorem and therefore by Proposition~\ref{morita_O_proposition} $\SF(kB)=\SF(kb)$. So for the rest of the proof we assume that $b$ is $G$-stable. By Lemma \ref{index_p_background_lem} $B$ is the unique block of $G$ covering $b$ and $G=ND$.
\newline
\newline
By~\cite[2.1]{el17} there exists a $G/N$-graded unit $a\in Z(kB)$ of multiplicative order a power of $p$ such that $kB=\bigoplus_{j=0}^{p-1}a^jkb$. Now suppose $\phi:kb\to kb^{(p^n)}$ is a strong Frobenius isomorphism. Then we can extend $\phi$ to an isomorphism $\tilde{\phi}:kB\to kB^{(p^n)}$ by sending $a$ to $\zeta_G^{\circ n}(a)$. By definition $\tilde{\phi}$ satisfies criterion (1) of Definition~\ref{strong_frobenius_definition}. Now suppose $S$ is a simple $kB$-module. Then by Schur's lemma $a$ must act by a scalar and since $a$ has order a power of $p$ it must act by the identity. Therefore $S\downarrow_N$ is also simple. So as $S\downarrow_N$ is associated to $S^{(p^n)}\downarrow_N$ through $\phi$, we must have $S$ associated to $S^{(p^n)}$ through $\tilde{\phi}$. We have shown that $\tilde{\phi}$ is a strong Frobenius isomorphism.
\end{proof}

We now prove that the strong $\cO$-Frobenius number behaves well when we quotient by central $p$-subgroups. If $G$ is a finite group and $Z\leq G$ a central $p$-subgroup then we define $\theta_Z$ to be the natural $\cO$-algebra homomorphism $\cO G\to\cO(G/Z)$. It is well known that $\theta_Z$ induces a bijection of blocks $B\mapsto B_Z$ of $\cO G$ and $\cO(G/Z)$.

\begin{prop}\label{central_p_subgroup}
Let $G$ be a finite group, $Z\leq G$ a central $p$-subgroup and $B$ a block of $\cO G$. Then $\SOF(B_Z)\leq\SOF(B)$.
\end{prop}

\begin{proof}
Let $\phi:B\to B^{(p^{\SOF(B)})}$ be a strong $\cO$-Frobenius isomorphism. We first note that $z^{(p^{\SOF(B)})}=z$ for all $z\in Z$ and so $(B_Z)^{(p^{\SOF(B)})}=(B^{(p^{\SOF(B)})})_Z$. Similarly $\phi(ze_B)=ze_{B^{(p^{\SOF(B)})}}$ for all $z\in Z$ and so $\phi$ induces a homomorphism $\phi_Z:B_Z\to B_Z^{(p^{\SOF(B)})}$. Adopting the notation from Lemma~\ref{lem:transfer}, suppose $g\in G$. Then $|C_G(g)|\theta_Z(C^G_g)=|C_{G/Z}(gZ)||Z|C^{G/Z}_{gZ}$. Therefore
\begin{align*}
\phi_Z(C^{G/Z}_{gZ}e_{B_Z})&=\phi_Z\circ\theta_Z\left(\frac{|C_G(g)|}{|C_{G/Z}(gZ)||Z|}C^G_ge_B\right)=\frac{|C_G(g)|}{|C_{G/Z}(gZ)||Z|}\theta_Z\circ\phi(C^G_ge_B)\\
&=\frac{|C_G(g)|}{|C_{G/Z}(gZ)||Z|}\theta_Z(C^G_{\zeta_G^{\circ\SOF(B)}(g)}e_{B^{(p^{\SOF(B)})}})=C^{G/Z}_{\zeta_{G/Z}^{\circ\SOF(B)}(gZ)}e_{B_Z^{(p^{\SOF(B)})}}\\
&=\zeta_{G/Z}^{\circ\SOF(B)}(C^{G/Z}_{gZ}e_{B_Z}),
\end{align*}
where the third equality follows from the fact that $\phi$ is a strong $\cO$-Frobenius isomorphism and the fourth from the fact that $C_G(g)=C_G(\zeta_G^{\circ\SOF(B)}(g))$ and that $C_{G/Z}(gZ)=C_{G/Z}(\zeta_G^{\circ\SOF(B)}(g)Z)$. Therefore, $\phi_Z$ satifies condition (1) of being a strong $\cO$-Frobenius isomorphism and condition (2) is automatic by Remark~\ref{remark_strong_frobenius} so the result is proved.
\end{proof}

We note that we were not able to prove this proposition with strong $\cO$-Frobenius number replaced by strong Frobenius number.


\section{Normal subgroups containing the defect group}

This section is devoted entirely to the proof of Corollary~\ref{powers_to_centre:cor} that is used when we prove Donovan's conjecture for abelian $2$-groups of order at most $64$ in Theorem~\ref{rank_at_most_four}. We begin with a Lemma proved in~\cite[Lemma 2.1]{ke04} but as we repeatedly use it in the rest of this section we state it here.

\begin{lem}\label{fixed_points_lemma}
Let $A$ be a finite dimensional $k$-algebra, $n$ a positive integer, $\phi:A\to A^{(p^n)}$ a $k$-algebra isomorphism and $A^\phi$ the subset of fixed points of $A$ under $A\to A$, $a\mapsto\phi(a)^{(p^{-n})}$. Then $A^\phi$ is an $\mathbb{F}_{p^n}$-subspace of $A$ such that $A=k\otimes_{\mathbb{F}_{p^n}}A^\phi$.
\end{lem}

\begin{thm}\label{normal_p'_index_theorem}
Fix a prime $p$. There exists a function $f:\mathbb{N}\times\mathbb{N}\times\mathbb{N}\to\mathbb{N}$ with the following property. Let $G$ be a finite group and $B$ a block of $kG$ with abelian defect group $D$. Let $N$ be a normal subgroup of $G$ containing $D$ and $b$ a block of $kN$ covered by $B$. Then there exists an isomorphism $\phi:B\to B^{(p^n)}$ for some positive integer $n\leq f(|D|,\SF(b),t)$, where $t$ is the dimension of the basic algebra of $b$. Furthermore we can choose $\phi$ such that $\phi(\alpha e_B)=\zeta_G^{\circ n}(\alpha e_B)$ for all $\alpha\in Z(kG)\cap kN$.
\end{thm}

\begin{proof}
We follow the proof of the Theorem at the beginning of $\S5$ in~\cite{ku95}. We first claim that we may assume that $B$ and $b$ share a block idempotent, that $p\nmid[G:N]$ and that $[G:N]$ can be bounded by a function of $|D|$.
\newline
\newline
Let $I_G(b)$ be the stablizer of $b$ in $G$ and $B'$ the unique block of $I_G(b)$ covering $b$ with Brauer correspondent $B$. Then $B'$ also has defect group $D$ and by~\cite{ku81} Theorem C and it's proof, $B\cong\Mat([G:I_G(b)],k)\otimes_kB'$ and from now on we identify the two sides of this isomorphism. The induced isomorphism
\begin{align*}
Z(B')&\cong Z(B)\\
\alpha&\mapsto1\otimes\alpha
\end{align*}
is given by the transfer map $\Tr_{I_G(b)}^G$ restricted to $Z(B')$. Next we note that if $\phi':B'\to B'^{(p^n)}$ satisfies the properties of the Theorem with respect to $N$ and $b$ then one can construct
\begin{align*}
\phi:B&\to B^{(p^n)}\\
E_{ij}\otimes x&\mapsto E_{ij}\otimes\phi'(x),
\end{align*}
for all $1\leq i,j\leq [G:I_G(b)]$, where $x\in B'$ and $E_{ij}$ is the $[G:I_G(b)]\times[G:I_G(b)]$ matrix with zeroes everywhere except in the $(i,j)^{\nth}$-entry where there's a $1$. Now for $\alpha\in Z(kG)\cap kN$ we have
\begin{align*}
\phi(\alpha e_B)&=\phi(\alpha\Tr_{I_G(b)}^G(e_{B'}))=\phi(\Tr_{I_G(b)}^G(\alpha e_{B'}))=\Tr_{I_G(b)}^G(\phi'(\alpha e_{B'}))=\Tr_{I_G(b)}^G(\zeta_{I_G(b)}^{\circ n}(\alpha e_{B'}))\\
&=\zeta_G^{\circ n}(\Tr_{I_G(b)}^G(\alpha e_{B'}))=\zeta_G^{\circ n}(\alpha\Tr_{I_G(b)}^G(e_{B'}))=\zeta_G^{\circ n}(\alpha e_B),
\end{align*}
where the fifth equality follows from Lemma~\ref{lem:transfer}. Therefore, we may assume $b$ is $G$-stable.
\newline
\newline
Now set $G[b]$ to be the normal subgroup of $G$ consisting of elements that induce inner automorphisms of $b$ and let $b'$ the unique block of $G[b]$ covered by $B$. Let $b_{\cO}$ and $b'_{\cO}$ be the corresponding blocks over $\cO$. Then by~\cite[Proposition 2.2]{kkl12} $b_{\cO}$ and $b'_{\cO}$ are Morita equivalent and so by Proposition~\ref{morita_O_proposition} we have that $\SF(b')=\SF(b)$. Also $D$ is a defect group for $b'$ and the basic algebras of $b$ and $b'$ are isomorphic. Therefore it is enough to prove the Theorem for $N$ replaced by $G[b]$. Now as in~\cite{ku95} we may assume that $B$ and $b$ share a block idempotent, that $p\nmid[G:N]$ and that $[G:N]$ can be bounded by a function of $|D|$.
\newline
\newline
Let $\psi:b\to b^{(p^{\SF(b)})}$ be a strong Frobenius isomorphism and $e\in b$ a basic idempotent of $b$. As $\psi(e)$ and $e^{(p^{\SF(b)})}$ are both basic idempotents of $b^{(p^{\SF(b)})}$ we may assume, by possibly composing $\phi$ with an inner automorphism of $b^{(p^{\SF(b)})}$, that $\psi(e)=e^{(p^{\SF(b)})}$. In other words $e$ is an element of $b^\psi$, in the sense of Lemma \ref{fixed_points_lemma}. As $e$ is a basic idempotent of $b$, Proposition~\ref{isom_centres} gives the isomorphism
\begin{align*}
Z(b)&\to Z(ebe)\\
z&\mapsto eze.
\end{align*}
Therefore we may identify $eZ(b)e$ and $Z(ebe)$. Now consider the isomorphism
\begin{align*}
\psi|_{Z(b)}=\zeta_N^{\circ\SF(b)}|_{Z(b)}:Z(b)\to Z(b^{(p^{\SF(b)})}).
\end{align*}
By a slight abuse of notation, we define the $\mathbb{F}_{p^{\SF(b)}}$-algebra $Z(b)^\psi$ to be $Z(b)^{\psi|_{Z(b)}}$, in the sense of Lemma~\ref{fixed_points_lemma}. Now since $e\in b^\psi$, $eZ(b)^{\psi}e=Z(ebe)^{\psi}$ is an $\mathbb{F}_{p^{\SF(b)}}$-subalgebra of $eb^\psi e=(ebe)^\psi$.
\newline
\newline
Now $eBe=ekGe$ is a crossed product of $G/N$ with $ebe=ekNe$. We claim that conjugation by $ege\in eBe$, for any $g\in G$ leaves $eZ(b)^\psi e$ invariant. We first note that
\begin{align}\label{align:conj}
(ege)(eze)=egze=e(gzg^{-1}g)e=(egzg^{-1}e)(ege),
\end{align}
for all $z\in Z(b)$, where the third equality follows from the fact that $Z(b)$ is invariant under conjugation by $g$. Also if $z\in Z(b)^\psi$ then
\begin{align}\label{align:conj2}
\zeta_N^{\circ\SF(b)}(gzg^{-1})^{(p^{-\SF(b)})}=(g\zeta_N^{\circ\SF(b)}(z)g^{-1})^{(p^{-\SF(b)})}=g\zeta_N^{\circ\SF(b)}(z)^{(p^{-\SF(b)})}g^{-1}=gzg^{-1}.
\end{align}
So by (\ref{align:conj}), in $(eBe)^\times$ we have
\begin{align*}
(ege)(eze)(ege)^{-1}=egzg^{-1}e,
\end{align*}
for all $z\in Z(b)^\psi$. Then by (\ref{align:conj2}) $egzg^{-1}e\in eZ(b)^\psi e$ and the claim is proved.
\newline
\newline
As described in~\cite[$\S2$]{ku95} $\Out(ebe)$, the group of outer $k$-algebra automorphisms of $ebe$ is an algebraic group. Furthermore
\begin{align}\label{align:action}
\Out(ebe)\times Z(ebe)&\to Z(ebe)\\
(\alpha\Inn(ebe),z)&\mapsto\alpha(z)\nonumber
\end{align}
is a well-defined algebraic group action. 
\newline
\newline
Now define $H$ to be the subgroup of $\Out(ebe)$ that leaves $eZ(b)^\psi e$ invariant. Note that $eZ(b)^\psi e$ is a finite and hence algebraic subset of $Z(ebe)$ and so $H$ is an algebraic subgroup of $\Out(ebe)$. We set 
\begin{align*}
e_0:=e, e_1:=\psi(e)\in b^{(p^{\SF(b)})}\text{ and }e_{i+1}:=\psi^{(p^{i\SF(b)})}(e_i)\in b^{(p^{(i+1)\SF(b)})},\text{ for }i\geq1.
\end{align*}
Note that, in fact, $e_i=e^{(p^{i\SF(b)})}$ for all $i\geq0$. We make identifications through the following isomorphisms:
\begin{align*}
\begin{CD}
e_0be_0
@>\psi>>
e_1(b^{(p^{\SF(b)})})e_1
@>\psi^{(p^{\SF(b)})}>>
e_2(b^{(p^{2\SF(b)})})e_2
@>\psi^{(p^{2\SF(b)})}>>
e_3(b^{(p^{3\SF(b)})})e_3
@>\psi^{(p^{3\SF(b)})}>>
\dots.
\end{CD}
\end{align*}
So we have a series of group homomorphisms $\theta_i:G/N\to H$ induced by the crossed products $e_i(B^{(p^{i\SF(b)})})e_i$ of $G/N$ with $e_i(b^{(p^{i\SF(b)})})e_i$. As in~\cite[$\S2$]{ku95}, up to conjugation in $H$ there are finitely many such homomorphisms. We claim that this finite number can be bounded by a function in $[G:N],\SF(b)$ and $t$. As stated above $Z(ebe)^\psi$ is an $\mathbb{F}_{p^{\SF(b)}}$-subalgebra of $(ebe)^\psi$ and $ebe$ has dimension $t$ and so the number of possible isomorphism classes of $ebe$ (and $eZ(b)e$ as a subalgebra) is bounded by a function in $\SF(b)$ and $t$, meaning so is the number of possible isomorphism classes of the algebraic group $H$. Furthermore, we are already assuming $[G:N]$ is bounded by a function in $|D|$ and hence the number of possible isomorphism classes of $G/N$ is also bounded by a function in $|D|$. This proves the claim. Therefore there exists some positive integer $m_0$, bounded by a function in $|D|$, $\SF(b)$ and $t$, and $h'\in H$ such that
\begin{align*}
\theta_{m_0}(gN)=h'\theta_0(gN)h'^{-1},
\end{align*}
for all $g\in G$. Moreover, the set of $\mathbb{F}_{p^{\SF(b)}}$-algebra automorphisms of $eZ(b)^\psi e$ can be bounded by a function in $\SF(b)$ and $t$ and so there exists some positive integer $m_1$, again bounded by a function in $\SF(b)$ and $t$, and $h\in H$ such that
\begin{align}\label{align:centre}
\theta_{m_0m_1}(gN)=h\theta_0(gN)h^{-1}\text{ and }\theta_{m_0m_1}(gN)(z)=\theta_0(gN)(z),
\end{align}
for all $g\in G$ and $z\in eZ(b)^\psi e$, where we are considering the action of $\Out(ebe)$ on $z\in eZ(b)e$ from (\ref{align:action}).
\newline
\newline
Next we see from~\cite[$\S3$]{ku95} that $|H^2(G/N,k)|$ is finite and hence can be bounded by a function in $|D|$. Therefore, by~\cite[$\S4$]{ku95} there exists a positive integer $m_2$, bounded by a function in $|D|$ such that $eBe$ and $e_{m_0m_1m_2}(B^{(p^{m_0m_1m_2\SF(b)})})e_{m_0m_1m_2}$ are weakly equivalent (in the sense of~\cite{ku95}) as crossed products of $G/N$ with $ebe$. In particular there exists a $k$-algebra isomorphism
\begin{align*}
\varphi:eBe\to e_{m_0m_1m_2}(B^{(p^{m_0m_1m_2\SF(b)})})e_{m_0m_1m_2}
\end{align*}
and if we follow through this isomorphism from~\cite{ku95} then (\ref{align:centre}) implies that $\varphi|_{eZ(b)e}$ is the identity (after we've identified $ebe$ and $e_{m_0m_1m_2}(b^{(p^{m_0m_1m_2\SF(b)})})e_{m_0m_1m_2}$ through $\psi^{\circ m_0m_1m_2\SF(b)}$). In other words
\begin{align*}
\varphi(eze)=e_{m_0m_1m_2}\zeta_N^{\circ m_0m_1m_2\SF(b)}(z)e_{m_0m_1m_2}
\end{align*}
for all $z\in Z(b)$ so certainly for all $z\in(Z(kG)\cap kN)e_B$. (Recall that $e_B=e_b$.) In particular $B$ and $B^{(p^{m_0m_1m_2\SF(b)})}$ are Morita equivalent with the isomorphism of centres satisfying $\alpha e_B\mapsto \zeta_G^{\circ n}(\alpha e_B)$ for all $\alpha\in Z(kG)\cap kN$.
\newline
\newline
Finally set $m_3:=|D|^2!$. Then $B$ and $B^{(p^{m_0m_1m_2m_3\SF(b)})}$ are Morita equivalent and, as in the proof of Proposition~\ref{sf_bound_with_iso}, $S$ and $S^{(p^{m_0m_1m_2m_3\SF(b)})}$ correspond through this Morita equivalence. So by Lemma~\ref{morita_implies_iso} $B$ and $B^{(p^{m_0m_1m_2m_3\SF(b)})}$ are isomorphic with the isomorphism of centres satisfying $\alpha e_B\mapsto \zeta_G^{\circ n}(\alpha e_B)$ for all $\alpha\in Z(kG)\cap kN$ and the claim is proved.
\end{proof}

\begin{rem}
We note that we would like to prove Theorem~\ref{normal_p'_index_theorem} with $\phi$ satisfying the strengthened condition that ``$\phi(\alpha e_B)=\zeta_G^{\circ n}(\alpha e_B)$ for all $\alpha\in Z(kG)$''. This would be a crucial step in proving Donovan's conjecture for blocks with abelian defect groups in characteristic $2$ as well as in proving a reduction to quasisimple groups for Donovan's conjecture for blocks with abelian defect groups in arbitrary characteristic.
\end{rem}

\begin{cor}\label{powers_to_centre:cor}
The function $f:\mathbb{N}\times\mathbb{N}\times\mathbb{N}\to\mathbb{N}$ from Theorem~\ref{normal_p'_index_theorem} satisfies the following property. Let $G$ be a finite group and $B$ a block of $kG$ with abelian defect group $D$. Let $L\lhd N\lhd G$, with $[G:N]=p$, $b_N$ the unique block of $kN$ covered by $B$ and $b_L$ a block of $kL$ covered by $b_N$ such that $b_L$ and $b_N$ share a defect group. Suppose there exists $g\in D\backslash N$ such that $g^p\in Z(N)\cap L$, then $\mf(B)\leq f(|D|,\SF(b_L),t)$, where $t$ is the dimension of the basic algebra of $b_L$.
\end{cor}

\begin{proof}
Theorem~\ref{normal_p'_index_theorem} tells us that we have an isomorphism $\phi:b_N\to b_N^{(p^n)}$ for some positive integer $n\leq f(|D|,\SF(b_L),t)$ such that $\phi(\alpha e_{b_N})=\zeta_G^{\circ n}(\alpha e_{b_N})$ for all $\alpha\in Z(kN)\cap kL$. We now proceed as in the proof of Theorem~\ref{normal_index_p_theorem}, noting that $B=\bigoplus_{i=0}^{p-1}g^ib_N$, that $g^p\in Z(kN)\cap kL$ and that we only need to construct an isomorphism $B\to B^{(p^n)}$ and not a strong Frobenius isomorphism.
\end{proof}


\section{Preliminary results for the application of the classification of finite simple groups}

We gather together some results which will be used in the proof of Theorem \ref{rank_at_most_four}.

\begin{lem}
\label{normal_defect}
Let $G$ be a finite group and $B$ a block of $\cO G$ with normal defect group $D\triangleleft G$. Write $n=|\Aut(D)|_{p'}$. Then $\SOF(B)\leq n^{\frac{1}{2} \log_2(n-1)}$.
\end{lem}

Note that we make no attempt to obtain an efficient bound.

\begin{proof}
By~\cite{ku85} we may assume that $G = D \rtimes \tilde{E}$ where $E=\tilde{E}/Z(\tilde{E})$ is the inertial quotient of $B$ and $Z(\tilde{E}) \leq [\tilde{E},\tilde{E}]$, so $|Z(\tilde{E})|$ is at most the order of the Schur multiplier of $E$. By~\cite[6.2]{gr56} the Schur multiplier of $E$ has order at most $|E|^{\frac{1}{2} \log_2(|E|-1)}$. The number of $p$-blocks of $G$ is $|E|$ and $|E| \leq |\Aut(D)|$, so we have the required bound.
\end{proof}

A block is \emph{inertial} if it is basic Morita equivalent to its Brauer correspondent in the normalizer of a defect group. A block $B$ of $kG$ is \emph{nilpotent covered} if there is some finite group $E$ with $G \lhd E$ and a nilpotent block of $kE$ covering $B$. Note that nilpotent covered blocks are inertial by~\cite{pu11}.

\begin{lem}
\label{nil_covered}
Let $G$ be a finite group with $N \lhd G$ such that $G/N$ is $p$-nilpotent and let $B$ be a block of $\cO G$ covering a $G$-stable inertial block $b$ of $N$. Then $\SF(B)\leq  n^{\frac{1}{2} \log_2(n-1)}$, where  $n = |\Aut(D \cap N)|_{p'}$. In particular this applies if $b$ is nilpotent covered.
\end{lem}

\begin{proof}
Let $M \lhd G$ be the preimage in $G$ of $O_{p'}(G/N)$. Then $G/M$ is a $p$-group. Let $b_M$ be a block of $\cO M$ covering $b$ and covered by $B$. Note that $b$ and $b_M$ both have defect group $D \cap N$. By the main theorem of~\cite{zh16} we have that $b_M$ is inertial, and so by Proposition~\ref{morita_O_proposition} $\SF(b_M)=\SF(C)$ for a block $C$ of a group with normal defect group isomorphic to  $D \cap N$. By Lemma \ref{normal_defect} $\SF(C) \leq n^{\frac{1}{2} \log_2(n-1)}$, where $n = |\Aut(D \cap N)|_{p'}$. Hence by Theorem~\ref{normal_index_p_theorem} we have that $\SF(B)\leq  n^{\frac{1}{2} \log_2(n-1)}$.
\end{proof}

\begin{lem}
\label{rank1or2}
Let $B$ be a block of $kG$ for a finite group $G$ with defect group $D \cong C_{2^n}$ or $C_{2^m} \times C_{2^n}$ for $m \neq n$. Then $B$ is nilpotent.
\end{lem}

\begin{proof}
Since $D$ is abelian, $B$ is nilpotent precisely when $N_G(D,b_D)=C_G(D)$ for each maximal $B$-subpair $(D,b_D)$. The result then follows since $\Aut(D)$ is a $2$-group.
\end{proof}

\begin{lem}
\label{index2nilpotent}
Let $B$ be a block of $G$ with abelian defect group $D$. 

(i) Suppose that $N \lhd G$ and $B$ covers a $G$-stable block $b$ of $N$ with defect group $D \cap N \cong C_{2^m} \times C_{2^m}$ for some $m \in \mathbb{N}$. If $b$ is not nilpotent, then $D \cap N$ is a direct factor of $D$.

(ii) Suppose that $D \cong C_{2^m} \times C_{2^m}$ for some $m \in \mathbb{N}$ and that $O_2(Z(G))$ is nontrivial and cyclic. Then $B$ is nilpotent.
\end{lem}

\begin{proof}
(i) Suppose that $D \cap N$ is not a direct factor of $D$. Then there is $g \in D \setminus N$ such that $g^2 \in (D \cap N) \setminus \{ 1 \}$. Write $H=\langle N,g \rangle G$. Since $H/N$ is a $2$-group, there is a unique block $B_H$ of $H$ covering $b$ and so by~\cite[15.1]{alp} $B_H$ has a defect group $D_H$ containing $g$. Hence $D_H \cong C_{2^{m+1}} \times C_{2^m}$. By Lemma \ref{rank1or2} $B_H$ is nilpotent and hence so is $b$ by~\cite[6.5]{kp90}.

(ii) Write $Z=O_2(Z(G))$, with order $2^l$ say. Now $B$ corresponds to a block $\bar{B}$ of $G/Z$ with defect group $D/Z \cong C_{2^m} \times C_{2^{m-l}}$. By Lemma \ref{rank1or2} $\bar{B}$ is nilpotent and hence so is $B$ by~\cite{wa94}.
\end{proof}

Recall that a \emph{component} of $G$ is a subnormal quasisimple subgroup of $G$. The components of $G$ commute, and we define the \emph{layer} $E(G)$ of $G$ to be the normal subgroup of $G$ generated by the components. It is a central product of the components. The \emph{Fitting subgroup} $F(G)$ is the largest nilpotent normal subgroup of $G$, and this is the direct product of $O_r(G)$ for all primes $r$ dividing $|G|$. The \emph{generalized Fitting subgroup} $F^*(G)$ is $E(G)F(G)$. A crucial property of $F^*(G)$ is that $C_G(F^*(G)) \leq F^*(G)$, so in particular $G/F^*(G)$ may be viewed as a subgroup of $\Aut(F^*(G))$. Background on this may be found in~\cite{asc00}.

\begin{prop}[Proposition 8.1 of~\cite{ekks14}]
\label{reduce}
Let $P$ be an abelian $2$-group. In order to prove Donovan's conjecture for blocks with defect group $D \cong P$, then it suffices to show that $\mf(B)$ is bounded in terms of $P$ for all blocks $B$ of finite groups $kG$ with defect group $D \cong P$ satisfying the following:

(i) $B$ is quasiprimitive, that is, for every normal subgroup $N$ of $G$, every block of $N$ covered by $B$ is $G$-stable;

(ii) $O_2(G)O_{2'}(G) = Z(G)=F(G)$;

(iii) $G$ is generated by the defect groups of $B$;

(iv) If $N \lhd G$ and $B$ covers a nilpotent block of $N$, then $N \leq Z(G)$;

(v) Every component of $G$ is normal in $G$;

(vi) If $N \leq G$ is a component, then $N \cap D \neq Z(N) \cap D$;

(vii) $DF^*(G)/F^*(G) \in \Syl_p(G/F^*(G))$.
\end{prop}

\begin{proof}
By~\cite[8.1]{ekks14} (which is an application of~\cite{du04} and~\cite{kp90}) in order to prove Donovan's conjecture it suffices to show that there are only finitely many Morita equivalence classes amongst blocks $B$ as in (i)-(vi). We show that (vii) follows from (i)-(vi). Since outer automorphism groups of simple groups are solvable, it follows that $G/F^*(G)$ is solvable. There is a chief series $F^*(G)=G_0 \lhd \cdots \lhd G_r=G$ such that each $G_{i+1}/G_i$ is either a $p$-group or a $p'$-group. Let $B_i$ be the unique block of $kG_i$ covered by $B$. (vii) follows by Lemma \ref{index_p_background_lem} applied to each $G_i \lhd G_{i+1}$ in turn.
\newline
\newline
It remains to show that it suffices to bound the Morita Frobenius number amongst these blocks. By~\cite[9.2]{ekks14} the entries of the Cartan matrices of blocks with defect group $D \cong P$ are bounded in terms of $P$. Hence by the proof of~\cite[1.4]{ke04} if the Morita Frobenius number of a collection of blocks with defect group isomorphic to $P$ is bounded in terms of $P$, then there are only finitely many Morita equivalence classes amongst these blocks, and we are done.
\end{proof}

\begin{defi}
For the purpose of this paper we call a pair $(G,B)$ \emph{reduced} if it satisfies the conditions of Proposition \ref{reduce}.
\end{defi}

The possible components are described in~\cite{ekks14}:

\begin{thm}[\cite{ekks14}]   
\label{component_classification}  
Let $G$ be a quasi-simple group.  If $B$ is a $2$-block of $\cO G$ with  abelian defect group $D$,  then one (or more) of the following holds:

(i) $G/Z(G)$ is one of  $A_1(2^a)$,   $\,^2G_2(q)$ (where $q \geq 27$ is a power of $3$ with odd exponent), or $J_1$,  $B$  is the principal block and $D$ is elementary abelian;

(ii) $G$  is $Co_3 $,   $B$ is a non-principal block, $D \cong C_2 \times C_2 \times  C_2 $ (there is one such block);

(iii) $B$ is a nilpotent covered block and $G/Z(G)$ is one of $PSL_n(q)$, $PSU_n(q^2)$, $E_6(q)$ or $^2E_6(q)$ for some $n \in \mathbb{N}$ and some power $q$ of an odd prime, and further $D$ cannot be elementary abelian of order $2^d$ with $d$ odd;

(iv) $G$ is of type $D_n(q)$ or $E_7(q)$, where $n=2t$ for $t$ odd and $q$ is a power of an odd prime. $B$ is Morita equivalent  to a block  $C$ of a   subgroup   $L=L_0 \times L_1$   of $G$   as  follows:   The  defect groups   of $C$ are isomorphic to  $D$,    $L_0$   is abelian  and the  $2$-block of  $L_1 $ covered by $C$ has  Klein $4$-defect groups.

(v) $B$ has Klein four defect groups.
\end{thm}

\begin{proof}
This is taken from Propositions 5.3 and 5.4 and Theorem 6.1 of~\cite{ekks14}. The statement in (iii) regarding blocks of odd defect comes from~\cite[5.4]{ekks14} and~\cite[12.1,13.1]{kkl12}.  
\end{proof}

We summarize information about the outer automorphism groups of the components here. Recall that the $p$-length of a $p$-solvable group $G$ is the smallest $m$ such that there is a normal series $1=G_0 \lhd \cdots \lhd G_n=G$ such that each $G_{i+1}/G_i$ is either a $p$-group or a $p'$-group and precisely $m$ factors $G_{i+1}/G_i$ are $p$-groups.

\begin{lem}
\label{outer}
Let $G$ be a non-abelian simple group. Then 

(i) $\Out(G)$ has $2$-length one unless $G=D_4(q)$ for $q$ a power of an odd prime;

(ii) $\Out(D_4(q)) \cong (C_2 \times C_2) \rtimes (C \times S_3)$ when $q$ is a power of an odd prime, where $C$ is a cyclic group;

(iii) If $G$ is of type $D_n(q)$ or $E_7(q)$, where $n=2t$ for $t$ odd and $q$ is a power of an odd prime, then $\Out(G)$ has a normal Sylow $2$-subgroup.
\end{lem}

\begin{proof}
This may be deduced directly from~\cite[2.5.12]{GLS}, noting in part (iii) that graph automorphims commute with field automorphisms.
\end{proof}


\section{Donovan's conjecture for abelian $2$-groups of rank at most four and for abelian $2$-groups of order at most $64$}
\label{rank4}

For $P$ a finite $p$-group write $\rk_p(P)$ for the rank of $P$, that is, the $p^{\rk_p(P)}$ is the order of the largest elementary abelian subgroup of $P$.

\begin{thm}
\label{rank_at_most_four}
Let $P$ a finite abelian $2$-group of rank at most $4$ or of order at most $64$. Then there are only finitely many Morita equivalence classes of blocks $B$ of group algebras $kG$ with defect group $D \cong P$.
\end{thm}

\begin{proof}
First of all, by~\cite{ekks14} the result holds for elementary abelian $2$-groups, so we may assume that $P$ is not elementary abelian and that $\rk_2(P)\leq 5$.
\newline
\newline
Since $\mf(B) \leq \SOF(B)$ for all blocks $B$, by Proposition \ref{reduce} it suffices to  show that $\SOF(B)$ is bounded in terms of $P$ for reduced pairs $(G,B)$ satisfying the given list of properties with defect group $D \cong P$. We show that: if $\rk_2(P) \leq 4$, then $\SOF(B)$ is at most the maximum of $(|P|)^2!$ and $|\Aut(Q)|_{2'}^{\frac{1}{2} \log_2 (|\Aut(Q)_{2'}|-1)}$ as $Q$ ranges over the subgroups of $P$; if $|P|=64$ and $\rk_2(P)=5$, then $\mf(B)$ is at most the maximum of $(128^2)!$, $f(64,(32^2)!,648)$ and $|\Aut(Q)|_{2'}^{\frac{1}{2} \log_2 (|\Aut(Q)_{2'}|-1)}$ as $Q$ ranges over the subgroups of $P$, for $f$ as in Theorem \ref{normal_p'_index_theorem}. These bounds are only given because we must show that $\SOF(B)$ is bounded in terms of the defect groups - they are not intended to be efficient bounds.
\newline
\newline
All blocks in the remainder of the proof are defined with respect to the discrete valuation ring $\cO$. In order to use the results of~\cite{ekks14} we require the assumption that $\cO$ contains a primitive $|P|^{\operatorname{th}}$ root of unity. 
\newline
\newline
Let $(G,B)$ be a reduced pair with $D \cong P$. Write $L:=E(G) =L_1 \cdots L_t$, where $L_1, \ldots,L_t$ are the components of $G$ and write $b_L$ for the unique block of $L$ covered by $B$, so $b_L$ has defect group $D \cap L$. By definition $L_i \lhd G$ for each $i$ and the $L_i$'s intersect pairwise in subgroups of $Z(G)$ (and so commute). Let $b_i$ be the unique block of $L_i$ covered by $B$ and write $\bar{b}_i$ for the unique block of $L_iO_2(G)/O_2(G)$ corresponding to $b_i$. Then $D_i : = D \cap L_i$ is a defect group for $b_i$ and $\bar{D}_i=D_iO_2(G)/O_2(G)$ is a defect group for $\bar{b}_i$.
\newline
\newline
We may assume that $t \geq 1$, otherwise $F^*(G)=Z(G)$ and $C_G(F^*(G)) \neq F^*(G)$ implies that $G=D$. We show that $ t \leq 2$. Now $O_2(L) \leq Z(G)$ and $(D \cap L)/O_2(L)$ is a defect group for a block of $L/O_2(L)$. Therefore $(D \cap L)/O_2(L)$ is a radical $2$-subgroup of $L/O_2(L)$ (recall that a $p$-subgroup $Q$ of a finite group $H$ is radical if $Q=O_p(N_H(Q))$ and that defect groups are radical $p$-subgroups). Hence $(D \cap L)Z(L)/Z(L)$ is a radical $2$-subgroup of $L/Z(L)\cong (L_1Z(L)/Z(L)) \times \cdots \times (L_tZ(L)/Z(L))$. By~\cite[Lemma 2.2]{ou95} it follows that $(D \cap L)Z(L)/Z(L) = \bar{R}_1 \times \cdots \times \bar{R}_t$, where $R_i = (D \cap L)Z(L)/Z(L) \cap L_iZ(L)/Z(L) \cong \overline{D}_i$ (and $R_i$ is a radical $2$-subgroup but not necessarily a defect group). If $\rk_2(R_i) \leq 1$ for some $i$, then $\bar{b}_i$ has cyclic defect group and so is nilpotent, hence $b_i$ is also nilpotent by~\cite{wa94} (where the result is stated over $k$, but follows over $\cO$ immediately), contradicting $(G,B)$ being a reduced pair. Hence $\rk_2(R_i)>1$ for all $i$. Hence since $D$ is abelian $\sum_{i=1}^t \rk_2(R_i) =\rk_2(D/O_2(G)) \leq \rk_2(D) \leq 5$ and we have $t \leq 2$. We treat the cases $t=1$ and $t=2$ separately.
\newline
\newline
Suppose that $t=1$, so $L$ is quasisimple. We consider the possibilities for $L$ (as described in Theorem \ref{component_classification}) in turn. Since $F^*(G)=LZ(G)$ we have $G/Z(G) \leq \Aut(LZ(G)/Z(G))$. In the following note that $O^{2'}(G)=G$, so if $G/L$ has a normal Sylow $2$-subgroup, then $G/L$ is a $2$-group and if $G/L$ has $2$-length one, then $G/L$ is $2$-nilpotent.
\newline
\newline
If $b_L$ is as in (i) or (ii) of Theorem \ref{component_classification}, then $b_L$ is Morita equivalent to a principal block and so $\SOF(b_L) \leq (|D \cap L|^2)!$ by Proposition \ref{sf_bound_with_iso}. By Lemma \ref{outer} $G/LZ(G)$, and so $G/L$, is a $2'$-group, hence $G/L$ must be trivial. So $\SOF(B) \leq (|P|^2)!$.
\newline
\newline
Suppose that $b_L$ is as in (iii), so is nilpotent covered. By Lemma \ref{outer} $G/LZ(G)$, and so $G/L$, has $2$-length one. Hence $G/L$ is $2$-nilpotent and so by Lemma \ref{nil_covered} $\SOF(B) \leq |\Aut(Q)|_{2'}^{\frac{1}{2} \log_2 (|\Aut(Q)|_{2'}-1)}$ where $Q=D \cap N$.
\newline
\newline
Suppose that $L$ is as in (iv), so $L$ is of type $D_n(q)$ or $E_7(q)$ and $b_L$ is Morita equivalent to a principal block. By Proposition \ref{sf_bound_with_iso} $\SOF(b_L) \leq (|D \cap L|^2)!$. By Lemma \ref{outer} $G/LZ(G)$, and so $G/L$, has a normal Sylow $2$-subgroup. Hence $G/L$ is a $2$-group and so by Theorem \ref{normal_index_p_theorem} $\SOF(B) \leq (|D \cap L|^2)! \leq (|P|^2)!$.
\newline
\newline
Suppose that $b_L$ has Klein four defect group. By Lemma \ref{index2nilpotent} $D \cap L$ is a direct factor of $D$. Suppose first that $L$ does not have type $D_4(q)$. Then by Lemma \ref{outer} $G/LZ(G)$, and so $G/L$, has $2$-length one. Hence $G/L$ has a normal $2$-complement. The unique block $c$ of $O^2(G)$ covered by $B$ has defect group $C_2 \times C_2$ and so is Morita equivalent to a principal block, so by Proposition \ref{sf_bound_with_iso} $\SOF(c) \leq 16!$ and so by Theorem \ref{normal_index_p_theorem} $\SOF(B) \leq 16! \leq (|P|^2)!$. Now suppose that $L$ is of type $D_4(q)$. By Lemma \ref{outer} there are normal subgroups $M$ and $N$ of $G$ with $L \leq M \leq N$ such that $M/L$ is a subgroup of $C_2 \times C_2$, $N/M$ is a $2'$-group and $G/N$ is a $2$-group. Let $b_M$, $b_N$ be the unique blocks of $\cO M$ and $\cO N$ covered by $B$ respectively. Now using Lemma \ref{index2nilpotent} $b_N$ has elementary abelian defect groups of rank at most $4$, so by~\cite{ea17} $b_N$ is Morita equivalent to a block $C$ of some finite group with $C^{(2)}=C$. Hence by Proposition \ref{morita_O_proposition} and Proposition \ref{sf_bound_with_iso} $\SOF(b_M) \leq (|D \cap M|^2)!$, and so by Theorem \ref{normal_index_p_theorem} $\SOF(B) \leq (|D \cap M|^2)! \leq (|P|^2)!$.
\newline
\newline
Now consider the case $t=2$.  Recall that $\rk_2(R_1), \rk_2(R_2) > 1$. Suppose first that $\rk_2(P) \leq 4$. Then by Lemma \ref{rank1or2} $b_i$ has defect group $D_i \cong C_{2^{n_i}} \times C_{2^{n_i}}$ for some $n_i \geq 1$. Since neither $b_1$ nor $b_2$ is nilpotent, and $L_i \lhd G$, by Lemma \ref{index2nilpotent} both $D_1$ and $D_2$ are direct factors of $D$ and $O_2(G)=1$. Hence $D =D_1 \times D_2$ and $D \leq L$, so $G=L$ by condition (iii) for $(G,B)$ to be a reduced pair. There is a finite group $E$ with $O_2(E)=1$ and $Z \leq Z(E)$ such that $E/Z=G$ and $E$ may be written $E=E_1 \times E_2$ where: (i) each $E_i$ is quasisimple; (ii) the unique block $B_E$ of $E$ with $Z$ in its kernel corresponding to $B$ is Morita equivalent to $B$; (iii) $B_E=b_{E_1} \otimes b_{E_2}$, where $B_{E_i}$ is a block of $E_i$ with defect group isomorphic to $D_i$. By~\cite[1.1]{ekks14} $B_{E_1}$ and $B_{E_2}$ are Morita equivalent to a principal block since they have homocyclic defect groups. Hence $B_E$ is also Morita equivalent to a principal block and so $\SOF(B) \leq \SOF(B_E) \leq (|P|^2)!$.
\newline
\newline
For the remainder of the proof suppose that $\rk_2(P) > 4$, hence we may assume $P \cong C_4 \times (C_2)^4$.
\newline
\newline
We treat the cases (a) $O_2(L)=1$ and (b) $O_2(L) \neq 1$ separately.

(a) Suppose that $O_2(L)=1$ (noting that we are not assuming that $O_2(G)=1$). It follows from Lemma \ref{index2nilpotent} that either $D \leq L$, or $D \cap L$ is elementary abelian of order $2^4$ or $2^5$.
\newline
\newline
If $D \leq L$, then $G=L$. After taking a suitable central extension by a $2'$-group as above we may assume that $G=L_1 \times L_2$. Then without loss of generality $b_1$ has defect group $(C_2)^2$ and $b_2$ has defect group $C_4 \times C_2$ (otherwise one of $b_1$, $b_2$ is nilpotent). By Theorem \ref{component_classification} $b_1$ and $b_2$ are Morita equivalent to principal blocks (if $b_2$ is nilpotent covered, then it is Morita equivalent to a principal block since by~\cite{pu11} nilpotent covered blocks are Morita equivalent to their Brauer correspondent, and the inertial quotient in this case is cyclic). Hence $B$ is Morita equivalent to a principal block and $\SOF(B) \leq (|P|^2)!$.
\newline
\newline
Suppose that $D \cap L \cong (C_2)^4$. We have $C_4 \cong DL/L \in \Syl_2(G/L)$ and $G/L$ is solvable, so $G/L$ has $2$-length one. Hence $G/L$ is $2$-nilpotent. Write $N=O^{2}(G)$. Then the unique block $b$ of $\cO$ covered by $B$ has defect group $(C_2)^4$. Hence by~\cite{ea17} $b_N$ is Morita equivalent to a block $C$ of some finite group with $C^{(2)}=C$. Hence by Proposition \ref{morita_O_proposition} and Proposition \ref{sf_bound_with_iso} $\SOF(b_N) \leq (|D \cap N|^2)!$, and so by Theorem \ref{normal_index_p_theorem} $\SOF(B) \leq (|D \cap N|^2)! \leq (|P|^2)!$.
\newline
\newline
Suppose that $D \cap L \cong (C_2)^5$. Then $C_2 \cong DL/L \in \Syl_2(G/L)$ and $G/L$ is solvable, so $G/L$ has $2$-length one and $G/L$ is $2$-nilpotent. Without loss of generality let $D_1 \cong (C_2)^2$ and $D_ 2 \cong (C_2)^3$. Write $\overline{G}=G/Z(L)$. By Theorem \ref{component_classification} and Lemma \ref{outer} $J:=\Aut(\overline{L}_2)$ has a chief series $\overline{L}_2 \lhd J_1 \lhd J$ such that $J_1/\overline{L}_1$ is a $2$-group and $J/J_1$ has odd order, and $H:=\Aut(\overline{L}_1)$ has a chief series $\overline{L}_1 \lhd H_1 \lhd H_2 \lhd H$ such that $H_1/\overline{L}_1$, $H/H_2$ are $2$-groups and $H_2/H_1$ has odd order (these factor groups may be trivial). With appropriate identifications we may write $G/Z(L) \leq H \times J$. Now $[G:L]_2=2$, so $\overline{G}/\overline{L}$ is $2$-nilpotent. If $O^2(G)$ were a direct product of groups, then we would be done in this case as the block of $O^2(G)$ covered by $B$ would be a tensor product of blocks with defect group $(C_2)^2$ and $(C_2)^3$. The following argument allows us to move to such a situation. 
\newline
\newline
Define $G_1$ to be the preimage of $\overline{G} \cap (H_1 \times \overline{L}_2)$ in $G$. If $G_1\neq L$, then $[G_1:L]=2$ and $G/G_1$ is a $2'$-group (and so $G=G_1$ since $(G,B)$ is a reduced pair). By~\cite{ea16} $b_L$ is Morita equivalent to a principal block, and so we are done in this case by Theorem \ref{normal_index_p_theorem}. Hence $G_1=L$. Let $G_2$  be the preimage of $\overline{G} \cap (H_2 \times \overline{L}_2)$ in $G$ and $G_3$ the preimage of $\overline{G}\cap (H \times J_1)$. Since $G/G_3$ is a $2'$-group, we must have $G=G_3$. Write $B_i$ for the unique block of $G_i$ covered by $B$. We examine $B_2$. There is a finite group $E=E_1 \times E_2$ and $Z \leq Z(E)$ such that $E/Z = G_2$, $E_1$ (resp. $E_2$) is a central extension of $\overline{G} \cap H_2$ (resp. $\overline{L}_2$) and there is a block $B_{E}$ of $\cO E$ Morita equivalent to $B_2$ and with isomorphic defect groups. Now $B_{E}$ is a tensor product of blocks with elementary abelian defect groups of order $4$ or $8$, so $B_{E}$ is Morita equivalent to a principal block by~\cite{ea16}. Hence by Proposition \ref{sf_bound_with_iso} $\SOF(B_2) \leq (|D \cap G_2|^2)!$. Now $G/G_2$ is a $2$-group and so by Theorem \ref{normal_index_p_theorem} $\SOF(B) \leq (|D \cap G_2|^2)! \leq (|P|^2)!$.
\newline
\newline
(b) Suppose $O_2(G) \neq 1$. Let $\overline{B}$ be the unique block of $\cO(G/O_2(L))$ corresponding to $B$. As usual there are several cases to consider. Suppose that $D \leq L$, so that $G=L$. In this case $D/O_2(L) \cong (C_2)^4$ or $(C_2)^5$. 
\newline
\newline
Suppose that $D/O_2(G) \cong (C_2)^4$. Then by~\cite{CEKL} $\overline{B}$ is source algebra equivalent to the principal block of $\cO (A_4 \times A_4)$, $\cO(A_4 \times A_5)$ or $\cO(A_5 \times A_5)$. Using the same argument as in the proof of~\cite[Theorem 1.1]{ea16} (where further details may be found), by~\cite[Corollary 1.14]{pu01} $B$ is Morita equivalent to the principal block of a central extension of $A_4 \times A_4$, $A_4 \times A_5$ or $A_5 \times A_5$, so $\SOF(B) \leq (|P|^2)!$.
\newline
\newline
Suppose that $D/O_2(G) \cong (C_2)^5$. There is a finite group $E$ and $Z \leq Z(E)$ such that $E/Z=G$ and $E$ may be written $E=E_1 \times E_2$ where: (i) each $E_i$ is quasisimple; (ii) there is a unique block $B_E$ of $\cO E$ corresponding to $B$ with $O_{2'}(Z)$ in its kernel; (iii) the unique block $\overline{B}_E$ of $E/O_2(Z)$ corresponding to $B_E$ is Morita equivalent to $B$; (iv) $B_E$ has defect group $D_E$ with $D_E/O_2(E) \cong (C_2^5)$; (v) $|O_2(E_i)| \leq 2$. Let $B_{E_i}$ be the unique block of $E_i$ covered by $B_E$. By Lemma \ref{index2nilpotent} $B_{E_1}$ has defect group $(C_2)^2$ or $(C_2)^3$, and $B_{E_2}$ has defect group $C_4 \times C_2 \times C_2$. Then $B_{E_1}$ is Morita equivalent to a principal block by~\cite{ea16} and $B_{E_2}$ is also Morita equivalent to a principal block by Theorem \ref{component_classification} (again noting that if it is nilpotent covered, then it is inertial and so Morita equivalent to the principal block of $A_4 \times C_4$ or $A_5 \times C_4$). Hence $B_E$ is Morita equivalent to a principal block and $\SOF(B_E) \leq (128^2)!$. By Proposition \ref{central_p_subgroup} $\SOF(B)=\SOF(\overline{B}_E) \leq \SOF(B_E)$ and we are done in this case.
\newline
\newline
Finally suppose that $O_2(G) \neq 1$ and $D \cap L \neq D$. Then $(D \cap L)/O_2(L) \cong (C_2)^4$. By Lemma \ref{index2nilpotent} applied to $\overline{B}$ there is $x \in D \setminus N$ of order $4$ such that $x^2 \in O_2(G)$. Now since $[G:L]_2=2$ and $G/L$ is solvable, it follows that $G/L$ is $2$-nilpotent. Hence we may apply Corollary \ref{powers_to_centre:cor}. Again applying~\cite{CEKL}, $b_L$ is Morita equivalent to the principal block of $A_4 \times A_4 \times C_2$, $A_4 \times A_5 \times C_2$ or $A_5 \times A_5 \times C_2$, so has basic algebra of dimension at most $648=2 \cdot 18^2$ (as the principal block of $A_5$ has basic algebra of dimension $18$), so $\mf(B) \leq f(64,(32^2)!,648)$ for $f$ as in Theorem \ref{normal_p'_index_theorem}.
\end{proof}

\end{document}